\documentclass{article}
\usepackage[preprint, nonatbib]{neurips_2020}
\usepackage{xcolor,eucal,enumerate,mathrsfs,enumitem}
\usepackage[normalem]{ulem}
\usepackage{amsmath,amssymb,epsfig,amsthm,tikz} 

\newcommand{ \ep }{\varepsilon}

\newcommand{ \Lexp}{L^{\rm exp}_{\ep}}
\newcommand{ \Lg}{L^{\rm \Phi}_{\ep}}

\newcommand{ \Dep}{D_{\ep}}

\newcommand{ \DNep}{D^N_{\ep}}

\newcommand{ \OT}{\operatorname{OT}}
\newcommand{ \OTep}{\operatorname{OT}_{\ep}}
\newcommand{ \OTepd}{\operatorname{\widetilde{OT}}_{\ep}}

\newcommand{ \OTNep}{\operatorname{OT}^N_{\ep}}
\newcommand{\Fam}{\mathcal{F}}

\newcommand{\Pro}{\mathcal{P}}

\newcommand{\R}{\mathbb{R}}

\newcommand{\gammaep}{\gamma^{\ep}}

\newcommand{\restr}[1]{\lower3pt\hbox{$|_{#1}$}}

\newcommand{\Fcep}{\mathcal{F}^{(c, \ep,\Phi)}}
\newcommand{\ucep}{u^{(c,\ep,\Phi)}}
\newcommand{\vcep}{v^{(c,\ep,\Phi)}}

\newcommand{\RR}{\mathbb{R}} 
\newcommand{\NN}{\mathbb{N}} 

\newcommand{\Ficep}{\mathcal{F}_i^{(N,c, \ep,\Phi)}}

\newcommand{\argmmax}{\mathop{\mathrm{argmax}}\limits}

\usepackage{color}
\usepackage{pstricks}
\usepackage{ifthen}
\newrgbcolor{simocolor}{.7 .1 .1}
\newrgbcolor{augucolor}{.1 .1 .7}

\theoremstyle{plain}
\newtheorem{theorem}{Theorem}[section]
\newtheorem{lemma}[theorem]{Lemma}
\newtheorem{proposition}[theorem]{Proposition}

\theoremstyle{definition}
\newtheorem{definition}{Definition}[section]

\theoremstyle{remark}

\newboolean{DEBUG}
\setboolean{DEBUG}{false}

\ifthenelse {\boolean{DEBUG}}
{}

\ifthenelse {\boolean{DEBUG}}
{}

\usepackage[utf8]{inputenc} 
\usepackage[T1]{fontenc}    
\usepackage{hyperref}       
\usepackage{url}            
\usepackage{booktabs}       
\usepackage{amsfonts}       
\usepackage{nicefrac}       
\usepackage{microtype}      

\title{Optimal Transport losses and Sinkhorn algorithm with general convex regularization}

\author{%
  Simone Di Marino \\
  Dipartimento di Matematica (DIMA) \\
  Universita di Genova\\
  Genova, Italy \\
  \texttt{simone.dimarino@unige.it} \\
  \And
  Augusto Gerolin \\
  Department of Theoretical Chemistry \\
  Vrije Universiteit Amsterdam\\
  Amsterdam, Netherlands \\
  \texttt{augustogerolin@gmail.com}
}

\begin{document}

\maketitle

\begin{abstract}
  
  We introduce a new class of convex-regularized Optimal Transport losses, which generalizes the classical Entropy-regularization of Optimal Transport and Sinkhorn divergences, and propose a generalized Sinkhorn algorithm. Our framework unifies many regularizations and numerical methods previously appeared in the literature. We show the existence of the maximizer for the dual problem, complementary slackness conditions,  providing a complete characterization of solutions for such class of variational problems. As a consequence, we study structural properties of these losses, including continuity, differentiability and provide explicit formulas for the its gradient. Finally, we provide theoretical guarantees of convergences and stability of the generalized Sinkhorn algorithm, even in the continuous setting. The techniques developed here are directly applicable also to study Wasserstein barycenters or, more generally, multi-marginal problems.
\end{abstract}





\section{Introduction}

In machine learning scenarios, one needs to compare two or more distributions 
supported on low-dimensional manifolds in much higher-dimensional spaces. The choice of the metric is crucial for tasks as regression, classification and generation, since the properties of the resulting average of a family of probability distributions (or barycenters) may vary significantly. 

Optimal Transport (OT) metrics
are better suited at capturing the geometry of the distribution than Euclidean distance, Kullback-Leibler or other
$f$-divergences. The topology induced  by the Wasserstein metric in the space of probability distributions is much weaker, having the ability, for example, to
handle measures with non-overlapping supports or supported in lower dimensional sets. 

Such features place OT as a promising tool in Machine Learning tasks, including generative modelling, classification and regression (e.g. \cite{AdlLun18, arjovsky2017, Bon16, Fro15, gulrajani17, SegCut15, TolBou18}). From a statistical and computational viewpoint, however, OT metrics still presents many challenges, particularly due to their instability, lack of smoothness as well as the difficulty to estimate their gradients in high dimension. 

Sinkhorn divergences introduced in \cite{GenCutPey} are based on the Shannon entropy-regularization of OT \cite{Cut} and provide a tractable method to train large scale generative models that can be be easily computed via Sinkhorn algorithm. 

In this paper, we introduce and study properties of discrete and continuous losses based on convex-regularized Optimal Transport.
\begin{equation}\label{intro:mainKL}
\OTep(\rho_1,\rho_2) = \inf_{\gamma\in \Pi(\rho_1,\rho_2)}\int_{\RR^d\times \RR^d}cd\gamma +\ep \int_{\RR^d\times \RR^d}\Phi(\gamma)d\left(\rho_1\otimes\rho_2\right),
\end{equation}
where $\rho_1,\rho_2$ are probability distributions, $c$ is a cost function (e.g. Euclidian distance), $\ep>0$ a regularization parameter and $\Phi$ a convex function.

Studying Optimal Transport problems with regularization other than the entropy is of interest since the resulting optimal transport plans are sparse, has a small support and are integrable in regular spaces \cite{LorMah19, LorManMey19}. Also, the approximation error of the primal problem are often smaller than with Shannon entropy regularization \cite{BloSegRol17}.

From a computational viewpoint, the generalized version of the Sinkhorn algorithm computing numerically \eqref{intro:mainKL} is introduced via the dual formulation. Understanding the structure of dual formulation of \eqref{intro:mainKL} is crucial to generalize the Sinkhorn algorithm and prove its convergence. As detailed in section \ref{sec:sinkhorn}, since $\Phi$ is not necessarily multiplicative as the entropy is, the solution optimal solution $\gammaep$ in \eqref{eq:mainKL} can not be decomposed in $\gammaep(x,y) = a(x)b(y)e^{-c/\ep}\rho_1\otimes\rho_2$ and therefore it is not possible to write explicitly an analogous Sinkhorn iterative system.  

Similarly to Sinkhorn-divergences, this new class of convex-regularized OT losses also interpolates between OT metric and Maximum Mean Discrepancies (MMD), which allow us, by choosing an appropriate range of the regularizer strenght $\ep$ to combine the favorable high-dimensional sample complexity of MMD with the geometric features of OT.

We point out the methods developed in this paper can be used to study convex-regularized Wasserstein Barycenters (e.g. \cite{AguCar,CutDou2014,DvuDviGasUriNed18,GaSw,LinHoCutJor19, Lui19,StaCaSol17}) or, more generally, multi-marginal optimal transport. This also includes a general version of the Sinkhorn algorithm adapted to the many-marginals case \cite{DMaGer19} and convex-regularizers $\Phi$. We introduce the setting and the main ideas in the Appendix \ref{app:multi}.

\subsection{Main contributions}

Many instances of $\Phi$ have been introduced in the literature previously, mainly for computational purposes, either to make the algorithm scale better with the number of points, or simply to induce stability. However, up to our knowledge, the use of these methods have not been studied from a theoretical standpoint in the current literature in both continuous and discrete settings. In this work, we formalize the convex-regularized OT problem in such a way that the convergence results observed in practice are reflected theoretically also. The algorithm structure is similar to that of the well-known Sinkhorn algorithm, although the proof strategy are not the same. 

The main constributions of this paper are:
\begin{itemize}
    \item We prove the existence of maximizers for dual problem of \eqref{intro:mainKL}, complementary slackness and, as a consequence, provide a complete characterization for the solutions of the primal problem \eqref{intro:mainKL} \, (ref. Theorem \ref{prop:equiv_comp}).
    \item We show that the convex-regularized losses $\OTep(\rho_1,\rho_2)$ are $weak^*$-continuous with respect to the marginals densities $\rho_1,\rho_2$ and, moreover, we derive explicit formulas for its gradient $\nabla \OTep(\rho_1,\rho_2)$ \, (ref. Proposition \ref{prop:OTcont}).  
    \item We introduce the generalized Sinkhorn algorithm that allow us to compute numerically \eqref{intro:mainKL} and prove convergence and stability (ref. Proposition \ref{prop:OTcont} and Theorem \ref{thm:convIPFP}). Our convergence results holds in both discrete and  continuous cases and are mesh-independent.
    \item Finally, we extend these results to convex-regularized Wasserstein Barycenter and more general multi-marginal optimal transport problems (ref. Appendix \ref{app:multi}).
\end{itemize}

\subsection{Related work} Numerical methods have been studied to solve OT losses with the problem \eqref{intro:mainKL} for different types of convex regularization \cite{CouRemTuiRak16, JanCutGra18,LorMah19,LorManMey19,MuzNocPatNei17}, including Sinkhorn-type methods for the Shannon entropy regularized case (e.g.  \cite{AltWeeRig17,Cut,DesPapRou18,GenCutPeyBac16,Lui18,Lui19}). In general, these methods exploit specific properties of the regularized function $\Phi$ itself and can not be extended to any convex function. A (generalized) conditional gradient approach used, for instance in \cite{FerPapPeyAuj14, RakFlaCou15}, could be applied in our case. However, the conditional gradient relies on iterative solving of a linearization, which can be slow in practice, or in the classical Sinkhorn algorithm, which requires to add an extra entropy-penalization in \eqref{intro:mainKL}. We refer to the above mentioned papers for a comparison between numerical methods based on Shannon entropy and different types of convex regularization.

Recently, also in \cite{LorMah19}  the problem of a generic convex-regularized optimal transport problem is tackled. Their setting differs from ours in two main points: they consider only the case when $X,Y$ are bounded Euclidean domains and continuous cost $c$, but they allow the penalization term to be the entropy with respect to a product measure of reference measures,which are possibly different from $\rho_0$ and $\rho_1$. While for the Entropy-regularized case assuming that the reference measures are $\rho_0$ and $\rho_1$ creates no harm (see \cite{DMaGer19}, section 1.4), in the case of general convex function we have no equivalence. This reflects into the fact that in \cite{LorMah19} the authors have to assume also some $L^{\infty}$ comparability between the reference measures and $\rho_0$ and $\rho_1$ and moreover the existence of dual potentials is in a weak space. While, in this work, we manage to get a priori estimates which allow us to prove also strong convergence for the Sinkhorn algorithm.

\section{Geometric losses for Machine Learning based on OT}
Let $(X,d_X)$ and $(Y,d_Y)$ be complete separable metric spaces, $\rho_1\in\mathcal{P}(X)$, $\rho_2\in \mathcal{P}(Y)$ be probability measures,and  $c:X\times Y\to \RR$ be a bounded cost function; additional hypotesis on the cost $c$ would imply better estimates for the optimal Kantorovich potential and better convergence, but we prefer to stick to the general case.

For the penalization part, let $\ep\geq0$ be a positive number and $\Phi:[0,+\infty]\to [0,+\infty]$ be a convex function, lower semi-continuous and superlinear at infinity; moreover we will assume also that $\Phi \in C^1((0, \infty))$ and $\Phi(1)=\Phi'(1)=0$ (or simply $\Phi$ be an Entropy). We point out that the assumption $\Phi(1)=\Phi'(1)=0$ is not restrictive since it can be obtained by simply adding a linear function to $\Phi$, which does not modify the minimizers in \eqref{eq:mainKL}.

We consider the \textit{Optimal Transport problem with general convex regularization}  
\begin{equation}\label{eq:mainKL}
\OTep(\rho_1,\rho_2) = \inf_{\gamma\in \Pi(\rho_1,\rho_2)}\int_{X\times Y}cd\gamma +\ep G(\gamma|\rho_1\otimes\rho_2),
\end{equation}
where $G:\mathcal{P}(X\times Y)\to\RR\cup \lbrace +\infty \rbrace$ is the $\Phi$-entropy of $\gamma$ with respect to $\rho_1 \otimes \rho_2$, that is
\[
G(\gamma|\rho_1\otimes\rho_2): =       \begin{cases} 
					\int_{X\times Y} \Phi\left(\dfrac{d\gamma}{d(\rho_1\otimes\rho_2)}\right) d\left(\rho_1\otimes\rho_2\right), \,  &\text{ if } \gamma \ll \rho_1\otimes \rho_2 \\
					 +\infty &\text{ otherwise }\\
     \end{cases}. \
\]
The notation $ \gamma \ll \rho_1\otimes \rho_2$  means that $\gamma$ is absolutely continuous with respect to $\rho_1\otimes \rho_2$, that is, $\gamma= \alpha \cdot \rho_1 \otimes \rho_2$, where $\alpha \in L^1(\rho_1 \otimes \rho_2)$ is called the density of $\gamma$ with respect to $\rho_1\otimes\rho_2$ and it is indicated by $\frac{d\gamma}{d(\rho_1\otimes\rho_2)}$ . 

The problem \eqref{eq:mainKL} is a natural generalisation of the celebrated \textit{entropic regularised of Optimal Transport}, which corresponds to the case  $\Phi(z) = z(\ln z-1)+1, \text{ for } z> 0$, introduced by Cuturi \cite{Cut}, see also \cite{Cut, CutPeyBook, DMaGer19, Idel16, LeoSurvey, Schr31} and references therein. In the particular case when $c(x,y) = d(x,y)^p, p\geq 1$ and $\ep = 0$, $W_p(\rho_1,\rho_2) = \sqrt[p]{\OT_{0}(\rho_1,\rho_2)}$ the problem \eqref{eq:mainKL} defines the $p$-Wasserstein distance  \cite{CutPeyBook, Vil03}. 

Besides the Shannon entropy, other types of convex functions $\Phi$ have been studied, including the quadratic regularization $\Phi(z) = \frac{1}{2}\vert z\vert^2$ \cite{EssSol18, LorMah19, LorManMey19} and the \textit{Tsallis entropy} $\Phi(z) = \frac{1}{p(p-1)} (z^p-p(z-1)),~p>1$ \cite{MuzNocPatNei17}. 

Similarly to the entropy-regularised Optimal Transport, the convex-regularised problem \eqref{eq:mainKL} can be express via its dual (Kantorovich) formulation.  We fist define the class the $\Lg$-spaces, which is the natural domain of the dual functional.


\begin{definition}[$\Lg$ spaces] Let $\ep>0$ be a positive number, $(X,d_X)$ be a complete separable metric space, $\rho$ be a probability measure in $X$, $\Phi$ an entropy and $\Psi= \Phi^*$ its Legendre conjugate. We define the set $\Lg(X,\rho)$ by  
\[
\Lg(X,\rho) = \left\lbrace u:X \to [-\infty, \infty[ \, : \,
u \text{ is measurable in } (X,\rho) \text{ and } \small{0<\int_X \Psi \left(u/\ep\right) \, d \rho < \infty}  \, \right\rbrace.
\]
\end{definition}
\paragraph{Dual functional:} Let $u\in \Lg(\rho_1), v\in \Lg(\rho_2)$ and consider the (Kantorovich) dual functional $\Dep:\Lg(\rho_1)\times \Lg(\rho_2)\to \R$
\begin{align}\label{eqn:dualdef}
\Dep(u,v) &= \int_X ud\rho_1 + \int_Y vd\rho_2    - \ep\int_{X\times Y} \Psi\left({\frac{u(x)+v(y)-c(x,y)}{\ep}}\right)d(\rho_1\otimes\rho_2).
\end{align}

We can defined the Kantorovich dual problem by
\begin{equation}\label{kantodual}
\sup \left\lbrace \Dep(u,v) : u\in\Lg(\rho_1) \text{ and } v\in\Lg(\rho_2) \right\rbrace.
\end{equation}

Although it is not always the case, closed-form expressions for the $\Psi$ are available, e.g. in the quadratic regularization and Tsallis entropy. For the Shannon entropy case we have \cite{Cut, GenChiBacCutPey,GenCutPeyBac16}
\[\Psi \left({\frac{u(x)+v(y)-c(x,y)}{\ep}}\right) = \exp\left({\frac{u(x)+v(y)-c(x,y)}{\ep}}\right).
\]

Note that if $u^+ \in L^1(\rho_1)$ and $v^+ \in L^1(\rho_2)$ then $\Dep(u,v)< \infty$. Moreover, as it will become clear later, we want to compute the supremum of $\Dep$, then in order to have $\Dep(u,v)>-\infty$ we need $u\in \Lg(\rho_1), v\in \Lg(\rho_2)$.

 We illustrate some ideas developed in this paper by analyzing the case of discrete distributions.
\subsection{An equivalent formulation in the discrete case} Let us for simplicity reduce to the case where $\hat{\rho_1}$ and $\hat{\rho_2}$ are discrete measures, i.e. finite sum of Diracs of the form $\hat{\rho_1} = \sum^I_{i=1}\rho^i_1\delta_{x_i}$ and $\hat{\rho_2} = \sum^J_{j=1}\rho^j_2\delta_{y_j}$, where $x_i\in X, \forall i\in I$ and $y_j\in Y, \forall j\in J$ and $\rho^i_1$ and $\rho^j_2$ are the weights respectively of $\rho_1$ and $\rho_2$. One may think of these discrete measure that they are sampling of a continuous measure or evaluations of continuous densities on a grid. In this setting the problem \eqref{eq:mainKL} can been rewritten via a cost matrix $c\in \RR_+^{I\times J}$ defined by $c_{i,j} = c(x_i,y_j)$: 
\begin{equation}\label{intro:KLdiscrete}
\OTep(\hat{\rho_1},\hat{\rho_2}) =\min \left\lbrace \sum_{i,j}c_{i,j}\gamma_{i,j} +\ep \sum_{i,j}\Phi\left(\frac{\gamma_{i,j}}{\hat{\rho}_{1,i}\hat{\rho}_{2,j}}\right)\hat{\rho}_{1,i}\hat{\rho}_{2,j} : \gamma 1_J = \hat{\rho}_1, \gamma^T 1_I = \hat{\rho}_2 \right\rbrace.
\end{equation}
Now, at least heuristically, we use Lagrange multipliers $u\in \RR^I$ and $v\in \RR^J$ to reinforce the marginal constraints. We thus consider the Lagrangian $L(\gamma,u,v)$
\begin{equation}
L(\gamma,u,v) = \sum_{i,j}c_{i,j}\gamma_{i,j} +\ep \sum_{i,j}\Phi\left(\frac{\gamma_{i,j}}{\hat{\rho_1}\hat{\rho_2}}\right)\hat{\rho_{1,i}}\hat{\rho_{2,j}} - u\cdot(\gamma 1_J - \hat{\rho_1}) + v\cdot( \gamma 1_I - \hat{\rho_2}). 
\end{equation}
In particular, the first order condition reads
\[
\dfrac{\partial L(\gamma,u,v)}{\partial \gamma_{i,j}} = 0 \quad \implies \quad  c_{i,j} + \ep(\Phi)'\left(\frac{\gamma_{i,j}}{\hat{\rho_1}\hat{\rho_2}}\right) - u_i - v_j = 0.
\]
Let is assume for a moment that $\Phi'$ is invertible, then $\gamma = (\Phi')^{-1}((u+v-c)/\ep)$ and so, since by Lagrange duality we have $(\Phi')^{-1} = \Psi'$, we get $\gamma = \Psi'((u+v-c)/\ep)\hat{\rho_1}\hat{\rho_2}$.

Notice that if we consider instead the problem with reference measures ${\mathfrak m}_1,{\mathfrak m}_2$ (e.g. the uniform measures or Lebesgue measure in the continuous setting), $\hat{\rho_1},\hat{\rho_2} \ll {\mathfrak m}_1,{\mathfrak m}_2$ \cite{LorMah19}
\begin{equation}\label{eq:m1m2}
\min \left\lbrace \sum_{i,j}c_{i,j}\gamma_{i,j} +\ep \sum_{i,j}\Phi\left(\frac{\gamma_{i,j}}{{\mathfrak m}_{1,i}{\mathfrak m}_{2,j}}\right){\mathfrak m}_{1,i}{\mathfrak m}_{2,j} : \gamma 1_J = \hat{\rho}_1, \gamma^T 1_I = \hat{\rho}_2 \right\rbrace,
\end{equation}
we then have that the optimal $\overline{\gamma}$ is solving \eqref{eq:m1m2} $\overline{\gamma} = \Psi'((u+v-c)/\ep){\mathfrak m}_1{\mathfrak m}_2$. As previously remarked, in the continuous case the relation between \eqref{intro:KLdiscrete} and \eqref{eq:m1m2} are not obvious for any convex-function $\Phi$ other than the Shannon entropy.

\paragraph{Computing gradients $\nabla \OTep(\hat{\rho_1},\hat{\rho_2})$:}in section \ref{sec:OTeploss} we show that the dual formulation \eqref{eqn:dualdef} admits a maximizer, which allow us in particular, to compute the gradient of the $\Phi-$Sinkhorn-loss $\nabla \OTep(\hat{\rho_1},\hat{\rho_2})$. In fact, the problem \eqref{intro:KLdiscrete} in discrete space is given by
\begin{equation}\label{intro:KLdiscretedual}
\OTep(\hat{\rho_1},\hat{\rho_2}) =\sup_{u,v} \left\lbrace \sum^I_{i=1}u_i\hat{\rho_{1,i}} + \sum^J_{j=1}v_j\hat{\rho_{2,j}} -\ep \sum_{i,j}\Psi\left(\frac{u_i+v_j-c_{i,j}}{\ep}\right)\hat{\rho_{1,i}}\hat{\rho_{2,j}} \right\rbrace.
\end{equation}
Assume for a moment that the supremum is archived by functions $\overline{u}$ and $\overline{v}$. Then, a direct computation shows that the gradient of $\OTep(\hat{\rho_1},\hat{\rho_2})$ exists and is given by 
\[
\nabla \OTep(\hat{\rho_1},\hat{\rho_2}) = \left(\sum_i\overline{u}_i-\ep\sum_{j}\Psi\left(\frac{\bar{u}_i+\bar{v}_j-c_{i,j}}{\ep}\right)\hat{\rho}_{2,j},\sum_j\overline{v}_j-\ep\sum_{i}\Psi\left(\frac{\bar{u}_i+\bar{v}_j-c_{i,j}}{\ep}\right)\hat{\rho}_{1,i}\right).
\]
In the continuous setting, the existence of maximizer for the dual problem \eqref{kantodual} and the existence of the gradient of $\OTep$ are not immediate (see Theorem \ref{thm:kanto2Nmax} and Proposition \ref{prop:OTcont}). The technical results that guarantee the existence of a maximizer in \eqref{kantodual} are stated in Lemmas \ref{lemma:F1F2} and \ref{lemma:betterpotentials}, below.

In fact, these results allow us to obtain a priori estimates for the maximizing sequences in \eqref{kantodual} via the $(c,\ep,\Phi)$-transforms defined in equations \eqref{eq:F1} and \eqref{eq:F2}. We anticipate that an analogous strategy is also used to prove the convergence of the Sinkhorn algorithm in section \ref{sec:sinkhorn} and Appendix \ref{sec:convergenceIPFP}.

\section{Characterization of the convex-regularized OT-loss via duality}\label{sec:OTeploss}

\subsection{A priori estimates and $(c,\ep,\Phi)$-transforms}
\label{sec:bounded}

\begin{definition}[$(c,\ep,\Phi)$-transform]
Let $(X,d_X)$, $(Y,d_Y)$ be complete separable metric spaces, $\ep>0$ be a positive number, $\rho_1 \in \mathcal{P}(X)$ and $\rho_2 \in \mathcal{P}(Y)$ be probability measures, $\Phi$ be an Entropy and let $c$ be a bounded cost on $X \times Y$. 
The $(c,\ep, \Phi)$-transform  $\Fcep:\Lg(\rho_1)\to L^0(\rho_2)$ is defined by
\begin{equation}\label{eq:F1}
 \Fcep ( u ) (y) \in \argmmax \lbrace \Dep(u,v) : u \in \Lg(X,\rho_1) \rbrace.
\end{equation}
Analogously, we define the $(c,\ep, \Phi)$-transform $\Fcep:\Lexp(\rho_2)\to L^0(\rho_1)$ by
\begin{equation}\label{eq:F2}
 \Fcep ( v ) (x) \in \argmmax \lbrace \Dep(u,v) : v \in \Lg(Y,\rho_2) \rbrace.
\end{equation}
Whenever it will be clear we denote $\vcep=\Fcep(v)$ and $\ucep=\Fcep(u)$, in an analogous way to the classical $c$-transform. 
\end{definition}

There are few cases when we can obtain explicity formulas for the $(c,\ep,\Phi)$-transform. For example if $\Phi$ is the Shannon entropy, the $(c,\ep,\Phi)$-transform corresponds to the \textit{SoftMin} operator \cite{DMaGer19,FeyFXVAmaPey}
\[
\ucep(y) = -\ep\ln\left(\int_{X}e^{(u(x)-c(x,y))/\ep}d\rho_1(x) \right).
\] 
In this case, it is easy to see that the $(c,\ep,\Phi)$-transform is consistent with the classical $c$-transform \cite{CutPeyBook,Vil03} when $\ep\to0^{+}$: $\ucep(y) = u^{c}(y) + O(\ep)$. 

\begin{lemma}\label{lemma:F1F2} Let $(X,d_X)$, $(Y,d_Y)$ be complete separable metric spaces, $\ep>0$, $c:X\times Y\to \RR$ and $u\in \Lg(\rho_1),v\in \Lg(\rho_2)$. Then 
\begin{itemize}
    \item[(i)] If $c$ is a bounded function, then $\ucep \in L^{\infty}(\rho_1)$ and $\vcep \in L^{\infty}(\rho_2)$. Moreover, ${\rm osc}(\ucep),{\rm osc}(\vcep)\leq 2\Vert c\Vert_{\infty}$. \item[(ii)] If $c$ is $L$-Lipschitz ($\omega$-continuous), then $\ucep,\vcep$ are $L$-Lipschitz ($\omega$-continuous).
\end{itemize}
\end{lemma}

\begin{proof} The strategy of the proofs for $(i)$ and $(ii)$ are similar, so we prove here second part of the statement and leave the part $(i)$ at Appendix \ref{app:apriori}. Moreover, we prove directly the stronger version for the $\omega$-continuity, since it implies Lipschitzianity. Assume $c$ is $\omega$-continuous. Let $u\in\Lg(\rho_1)$, $v\in\Lg(\rho_2)$ and $\Psi = \Phi^*$. By definition, we have that  $\ucep(y)$ can be defined pointwisely by
\[
\ucep (y) :=\argmmax_{t\in\RR} \, \alpha^y(t)  \qquad \text{ where } \qquad \alpha^y(t) := t  - \ep\int_{X}\Psi((u+t-c)/\ep).
\]
Notice that since $\Phi \in C^1$ we have that $\Psi$ is strictly concave; in particular $\alpha^y(t)$ is also strictly concave and  $(\alpha^y)'(t_0) = 0$, if and only if, 
\[
1  = \int_{X}\Psi'((u(x)+t_0-c(x,y))/\ep)d\rho_1(x) =:  \beta^{y}(t_0). 
\]
Then we can define $\ucep$ as the unique function such that
$$1 = \int_{X}\Psi'((u(x)+\ucep(y)-c(x,y))/\ep)d\rho_1(x), \quad \forall~ y\in Y.$$

Since $\Psi'$ is a increasing function in $t$, $\beta^y(t)$ is also increasing in $t$. Assume that  $t_{y}$ and $\tilde{t}_{\tilde{y}}$ are such that $\beta^{y}(t_y) = \beta^{\tilde{y}}(\tilde{t}_{\tilde{y}})$, then we have
\begin{align*}
\beta^{y}(t_y) &= \int_{X}\Psi'((u(x)+t_y-c(x,y))/\ep))d\rho_1(x) \\
&\geq \int_{X}\Psi'((u(x)+t_y-c(x,\tilde{y})-\omega(d_Y(y,\tilde{y})))/\ep)d\rho_1(x) = \beta^{\tilde{y}}(\tilde{t}_{\tilde{y}}-\omega(d_Y(y,\tilde{y})),
\end{align*}
which implies that $t_y \geq \tilde{t}_{\tilde{y}}-\omega(d_Y(y,\tilde{y}))$. By exchanging the roles of $\tilde{t}_{\tilde{y}}$  we conclude that $\vert \tilde{t}_{\tilde{y}} -  t_y\vert \leq \omega(d_Y(y,\tilde{y}))$. But taking $\ucep(y)= t_y$ and $\ucep(\tilde{y}) = \tilde{t}_{\tilde{y}}$, we get that $\ucep$ is $\omega$-continuous. 
\end{proof}
The next two lemmas shows that if $(u,v)$ are admissible for the dual problem, then there exist a better couple of potentials $(u^*,v^*)$ obtained via the $(c,\ep,\Phi)$-transform that increase the value of $\Dep$ in \eqref{eqn:dualdef} and are more regular. The proof of Lemmas \ref{lemma:dual} and \ref{lemma:betterpotentials} are in Appendix \ref{app:apriori}.

\begin{lemma}\label{lemma:dual} Let $(X,d_X)$, $(Y,d_Y)$ be complete separable metric spaces, $\ep>0$ be a positive number, $\rho_1 \in \mathcal{P}(X)$ and $\rho_2 \in \mathcal{P}(Y)$ be probability measures, $\Phi$ be an Entropy, $\Psi = \Phi^*$, $c$ be a bounded cost on $X \times Y$ and let us consider $\Dep: \Lg(\rho_1) \times \Lg(\rho_2)\to\mathbb{R}$ defined as in \eqref{eqn:dualdef}. Then
\begin{equation}\label{est:optcond}
D_{\ep}(u,\ucep) \geq  D_{\ep}(u,v), \forall ~v \in \Lg(\rho_2), \text{ and } D_{\ep}(u,\ucep) =  D_{\ep}(u,v) \text{ iff } v = \ucep.
\end{equation}
\end{lemma}

\begin{lemma}\label{lemma:betterpotentials} Let us consider $u \in \Lg(\rho_1)$ and $v \in \Lg(\rho_2)$. Then there exist $u^* \in \Lg(\rho_1)$ and $v^* \in \Lg(\rho_2)$ such that
\[
D_{\ep}(u,v) \leq D_{\ep}(u^*,v^*), \quad \text{ and } \quad \| u^* \|_{\infty},\| v^* \|_{\infty} \leq 2\| c\|_{\infty}.
\]
Moreover we can choose $a \in \mathbb{R}$ such that $u^*= (v+a)^{(c,\ep,\Phi)}$ and $v^*=(u^*)^{(c,\ep,\Phi)}$.
\end{lemma}

\subsection{Existence of a maximizer for the dual problem and complementary slackness}\label{sec:existmax}

\begin{theorem}\label{thm:kanto2Nmax} Let $(X,d_X)$, $(Y,d_Y)$ be complete separable metric spaces, $c:X\times Y\to \RR$ be a bounded cost, $\rho_1 \in \mathcal{P}(X)$, $\rho_2 \in \mathcal{P}(Y)$ be probability measures and $\ep>0$ be a positive number. Consider the problem
\begin{equation}\label{eq:dualitySep}
\sup \left\{ D_{\ep}(u,v) \; : \; u \in \Lg(\rho_1) , v \in \Lg(\rho_2) \right\}.
\end{equation}
Then the supremum in \eqref{eq:dualitySep} is attained for a unique couple $(u_0,v_0)$ (up to the trivial tranformation $(u,v) \mapsto (u+a,v-a)$). In particular we have $u_0 \in L^{\infty}(X,\rho_1)$ and  $v_0 \in L^{\infty}(Y,\rho_2).$
\end{theorem}

\begin{proof}
Now, we are going to show that the supremum is attainded in the right-hand side of \eqref{eq:dualitySep}. Let $(u_n)_{n\in\NN} \subset \Lg(\rho_1) $ and $(v_n)_{n\in \NN} \subset \Lg(\rho_2)$ be  maximizing sequences. Due to Lemma  \ref{lemma:F1F2}, we can suppose that $u_n\in L^{\infty}(\rho_1)$, $v_n\in L^{\infty}(\rho_2)$ and $\Vert u_n\Vert_{\infty},\Vert v_n\Vert_{\infty} \leq 2 \| c \|_{\infty}$. Then by Banach-Alaoglu theorem there exists subsequences $(u_{n_k})_{n_k\in\NN}$ and $(v_{n_k})_{n_k\in\NN}$ such that $u_{n_k}\rightharpoonup \overline{u}$ and $v_{n_k}\rightharpoonup \overline{v}$. In particular, $\tilde{u}_{n_k}+\tilde{v}_{n_k}-c\rightharpoonup \overline{u}+\overline{v}-c$. 

First, notice that since $t \mapsto \Psi(t)$ is a convex function ($\Psi = \Phi^*$), we have
\begin{align*}
\liminf_{n\to\infty} \int_{X\times Y}\Psi\left(\frac{u_n+v_n-c}{\ep}\right)d(\rho_1\otimes\rho_2) &= \liminf_{n\to\infty} \int_{X\times Y}\Psi\left(\frac{u_n+v_n-c}{\ep}\right)d(\rho_1\otimes\rho_2) \\
&\geq  \int_{X\times Y}\Psi\left(\frac{\overline{u}+\overline{v}-c}{\ep}\right)d(\rho_1\otimes\rho_2).
\end{align*}
Moreover,
\begin{align*}
\sup_{u,v} \Dep(u,v) &= \lim_{n\to\infty}\left\lbrace\int_X u_nd\rho_1 + \int_Y v_nd\rho_2 - \ep\int_{X\times Y}\Psi\left(\frac{u_n+v_n-c}{\ep}\right) d(\rho_1\otimes\rho_2) \right\rbrace \\
&\leq \lim_{n\to\infty}\left\lbrace\int_X u_n d\rho_1 + \int_Y v_n d\rho_2  \right\rbrace - \ep \liminf_{ n \to \infty}\left\lbrace \int_{X\times Y}\Psi\left(\frac{u_n+v_n-c}{\ep}\right)d(\rho_1\otimes \rho_2)\right\rbrace \\
&\leq \int_X \overline{u}d\rho_1 + \int_Y \overline{v}d\rho_2 - \ep\int_{X\times Y}\Psi\left(\frac{\overline{u}+\overline{v}-c}{\ep}\right)d(\rho_1\otimes\rho_2) = \Dep(\overline{u},\overline{v}).
\end{align*}

So, $(\overline{u},\overline{v})$ is a maximizer for $\Dep$. By construction, we have also that $\overline{u} \in L^{\infty}(\rho_1)$ and $ \overline{v} \in L^{\infty}(\rho_2)$. Finally, the strictly concavity of $D_{\ep}$ and Lemma \ref{lemma:dual} implies that the maximizer is unique and, in particular $\overline{v} = \overline{u}^{(c,\ep,\Phi)}$. 
\end{proof}

By using the existence of maximizer potentials $(u_0,v_0)$, one can obtain directly the duality between \eqref{eq:mainKL} and \eqref{eqn:dualdef} as well as a characterization of the primal problem \eqref{eq:mainKL}, see Appendix \ref{app:existmax} for the proof.

\begin{proposition}\label{prop:dual} 
Let $\ep>0$ be a positive number, $(X,d_X)$ and $(Y,d_Y)$ be complete separable metric spaces, $c:X\times Y\to \RR$ be a Borel cost function, $\rho_1 \in \mathcal{P}(X)$, $\rho_2 \in \mathcal{P}(Y)$ be probability measures. Then for every $\gamma \in \Pi(\rho_1,\rho_2)$, $u\in \Lg(\rho_1)$ and $v\in \Lg(\rho_2)$ we have 
\[
\Dep(u,v) \leq \OTep(\rho_1,\rho_2), \text{ with equality if and only if } \gamma = \Psi'((u+v-c)/\ep)\rho_1\otimes\rho_2.
\]
\end{proposition}
Finally, we are able to the main main theorem of this section. The proof is in Appendix \ref{app:existmax}.

\begin{theorem}[Equivalence and complementarity condition]\label{prop:equiv_comp}
Let $\ep>0$ be a positive number, $(X,d_X)$ and $(Y,d_Y)$ be complete separable metric spaces, $c:X\times Y\to \mathbb{R}$ be a bounded cost function, $\Phi$ be an Entropy, $\Psi = \Phi^*$, $\rho_1 \in \mathcal{P}(X)$, $\rho_2 \in \mathcal{P}(Y)$ be probability measures. Then given $u^* \in \Lg(\rho_1) , v^* \in \Lg(\rho_2)$, the following are equivalent:
\begin{enumerate}
\item \emph{(Maximizers)} $u^*$ and $v^*$ are maximizing potentials for \eqref{eq:dualitySep};
\item \emph{(Maximality condition)} $\Fcep (u^*)=v^*$ and $\Fcep(v^*)=u^*$ (in particular $u^*\in L^{\infty}(\rho_1)$ and $v^*\in L^{\infty}(\rho_2)$);
\item \emph{(Complementary Slackness)} let $\gamma^*= \Psi'\left((u^*(x)+v^*(y)-c(x,y))/\ep\right) \cdot \rho_1 \otimes \rho_2$, then $\gamma^* \in \Pi(\rho_1, \rho_2)$;
\item \emph{(Duality)} $\OTep(\rho_1,\rho_2) = D_{\ep} (u^*,v^*)$.
\end{enumerate}
Moreover in those cases $\gamma^*$, as defined in 3, is also the (unique) minimizer for the problem \eqref{eq:mainKL}.
\end{theorem}

\subsection{Properties of $\OTep$ and $\Phi$-Sinkhorn divergences}\label{sec:divergences}

We start by showing the continuity of the functional $\OTep$ with respect to the marginals. The proofs of Proposition \ref{prop:OTcont} and Proposition \ref{prop:sinkdivlimits} below are given in the Appendix \ref{app:divergence}.

\begin{proposition}\label{prop:OTcont}
The functional $\OTep$ is continuous: if $(\rho^n_1)_{n\in\NN} \subset \mathcal{P}(X)$ and $(\rho^n_2)_{n\in\NN} \subset \mathcal{P}(Y)$ are sequences weakly converging respectively to $\rho_1$ and $\rho_2$, then the corresponding Kantorovich potentials $(u_n)_{n\in \NN}$ and $(v_n)_{n\in\NN}$ converges uniformly in $L^{\infty}$ to $(\overline{u},\overline{v})$.  Moreover, gradient of $\OTep(\rho_1,\rho_2)$ is given by 
\[
\nabla \OTep(\rho_1,\rho_2) = \left(\overline{u} - \int_{Y}\Psi\left(\frac{\overline{u}+\overline{v}-c}{\ep}\right)d\rho_2, \overline{v} - \int_{X}\Psi\left(\frac{\overline{u}+\overline{v}-c}{\ep}\right)d\rho_1  \right),
\]
where $(\overline{u},\overline{v})$ are Kantorovich potentials such that $\overline{u} = \vcep$ and $\overline{v} = \ucep$.
\end{proposition}

In analogy to the Sinkhorn divergences \cite{FeyFXVAmaPey, GenCutPeyBac16}, we can define a new class of Statistical Divergences based in \eqref{intro:mainKL}: $\OTepd(\rho_1,\rho_2) = \OTep(\rho_1,\rho_2) - \frac{1}{2}\left(\OTep(\rho_1,\rho_1) + \OTep(\rho_2,\rho_2)\right)$. It would interesting to prove that $\OTepd \geq 0$; in the Shannon entropy case \cite{FeyFXVAmaPey} the authors prove that under an assumption of positivity (in the sense of kernels) of $e^{-c /\ep}$, the Sinkhorn divergence is in fact positive. However the method used does not generalize to different types of entropies, and so this remains an open question.

As for the Shannon entropy, the $\Phi$-Sinkhorn divergence interpolates between the Optimal Transport distance and the Maximum Mean Discrepancy with kernel $c$.

\begin{proposition}\label{prop:sinkdivlimits}
The $\Phi$-Sinkhorn divergence has the following limits
\begin{itemize}
\item[(a)] $\OTepd(\rho_1,\rho_2) \to \OT_0(\rho_1,\rho_2)$, when $\ep \to 0^+$.
\item[(b)] $\OTepd(\rho_1,\rho_2) \to \frac 12 \|\rho_1-\rho_2\|_{-c} := -\frac 12 \int_{X\times Y} c\, d((\rho_1-\rho_2)\otimes(\rho_1-\rho_2))$, when $\ep \to +\infty$.
\end{itemize}
Moreover, any sequence $\gammaep$ solving \eqref{eq:mainKL} weakly converges to a solution of the Optimal Transport problem (i.e, $\ep=0$ in \eqref{eq:mainKL}) when $\ep\to0$ and $\gammaep\to \rho_1\otimes\rho_2$ when $\ep\to+\infty$.
 \end{proposition}

\section{Computational Scheme: generalized Sinkhorn algorithm }\label{sec:sinkhorn}

In this section, we introduce the generalized \textit{Iterative Proportional Fitting Procedure} (IPFP) or \textit{Sinkhorn algorithm} \cite{Sin64} to compute the problem \eqref{eq:mainKL}. Our method can be extend to prove convergence to the corresponding Sinkhorn algorithm also in the multi-marginal case. The setting is presented in the Appendix \ref{app:multi}.

The main idea of the Sinkhorn algorithm is to construct the measure $\gammaep \in \Pi(\rho_1,\rho_2)$ realizing minimum in \eqref{eq:mainKL} by fixing the shape of an ansatz as $\gammaep_{n} = \Psi'\left((u^{n}(x)+v^{n}(y)-c(x,y)) /\ep\right) \rho_1 \otimes \rho_2$ (since this is the actual shape of the minimizer due to Theorem \ref{prop:equiv_comp}) and then alternatively updating either $u^n$ or $v^n$, by matching one of the marginal distribution respectively to the target marginals $\rho_1$ or $\rho_2$. Fixing for example $u^{n}$, this amounts to find $v^n$ such that 
\begin{equation}\label{eqn:IPFPsing}\int_X \Psi'\Bigl(\frac{u^n(x)+v^{n}(y)-c(x,y)}{\ep}\Bigr) \, d \rho_1(x) =1.
\end{equation}
This is an implicit definition for $u^n$ and, depending on the shape of $\Psi$, one can hope to solve the equation for $v^n(y)$ explicitely, but in general this is not possible.

We recall however from the proof of Lemma \ref{lemma:F1F2}, that solving Equation \eqref{eqn:IPFPsing} actually amounts to solve the maximization of the strictly concave function $\alpha^y$. This is useful in practice since we can use a (trivially parallelizable) Newton algorithm to find $v^n(y)$, but we can also observe that we have $v^n = (u^n)^{(c,\ep,\Phi)}$. In particular, as in the case of the \emph{classical} Sinkhorn algorithm, also this generalized IPFP can be seen as an alternate maximization procedure.

In other words, the generalized IPFP sequences $(u^n)_{n\in\NN}$ and $(v^n)_{n\in\NN}$ can defined by $v^0(y) = 1$ and
\begin{equation}\label{eq:IPFPiteration}
\begin{array}{lcl}
u^n(x) & = & \argmmax_{u\in\Lg(\rho_1)}\left\lbrace \int_{X}u d\rho_1 -\ep\int_{X\times Y}\Psi\left(\frac{u+v^{n-1}-c}{\ep}\right)d(\rho_1\otimes\rho_2) \right\rbrace, \\
v^n(y) & = & \argmmax_{v\in\Lg(\rho_2)}\left\lbrace \int_{X}vd\rho_2 -\ep\int_{X\times Y}\Psi\left(\frac{u^{n}+v-c}{\ep}\right)d(\rho_1\otimes\rho_2) \right\rbrace. \\
\end{array}
\end{equation}

\paragraph{Example (Entropy-regularized Sinkhorn algorithm):} When $\Phi(z)=z \ln (z)-z$ is the relative entropy, we have $\Psi'(t)=e^t$ and so Equation \eqref{eqn:IPFPsing} can be solved explicitly: we have
\[
\begin{array}{lcl}\label{eq:IPFPentcep}
 v^n(y) = & - \ep\log\left(\int_X e^{(c(x,y)-u^{(n-1)}(x))/\ep}d\rho_1\right) = (u^{(n-1)})^{(c,\ep,\Phi)},  \\
 u^n(x) =&-\ep\log\left(\int_Y e^{(c(x,y)-v^{(n-1)}(x))/\ep}d\rho_2\right) = (v^n)^{(c,\ep,\Phi)}.
\end{array}
\]
Via the new variables $a_n := \exp(u_n/\ep)$ and $b_n ;= \exp(v^n/\ep)$ one can then rewrite the system in the (usual) form: $a^0(x) = 1$,
\begin{equation}\label{eq:IPFPent}
 b^n(y)  =  \dfrac{1}{\int e^{-c(x,y)/\ep} a^{n-1}(x)d\rho_1(x)} \quad \text{and} \quad a^n(y)  =  \dfrac{1}{\int e^{-c(x,y)/\ep} b^{n}(y)d\rho_2(y)}. 
\end{equation}

\paragraph{Theoretical guarantees of convergence:} The original approach of showing converging of the usual Entropic Sinkhorn-algorithm, due to Franklin and Lorenz \cite{FraLor89} (see also \cite{CheGeoPav, RusIPFP} for the continuous case), are based on a fixed-point contraction theorem of the iterates \eqref{eq:IPFPent} under the Hilbert metric. Such approach can not be easily applied in our case since the regularization $\Phi$ is not necessarily multiplicative as, for instance, the Shannon entropy is. Hence, the solution optimal solution $\gammaep$ in \eqref{eq:mainKL} generally can not be decomposed in $\gammaep(x,y) = a(x)b(y)e^{-c/\ep}\cdot\rho_1\rho_2$.

The main strategy of our proof instead are based on ideas from optimal transport theory: we prove a priori estimates and deduce compactness results for the $(c,\ep,\Phi)$-transforms, which can guarantee that the Sinkhorn iteratations \eqref{eq:IPFPiteration}, seen as an alternate maximization in the dual problem, are strongly converging. This is also consistent with algorithms proposed to solve the quadratic-regularized \cite{LorMah19} and Shannon entropy cases \cite{ChiPeySchVia16,KarlRin17}. The proof of the next theorem is given in the Appendix \ref{sec:convergenceIPFP}.

\begin{theorem}\label{thm:convIPFP}
Let $(X,d_X)$ and $(Y,d_Y)$ be complete separable metric spaces, $c:X\times Y\to \RR$ be a Borel bounded cost, $\Phi$ be a entropy function, $\Psi = \Phi^*$, $\rho_1 \in \mathcal{P}(X)$ and $\rho_2\in\mathcal{P}(Y)$ be probability measures. If $(u^n)_{n\in\NN}$ and $(v^n)_{n\in\NN}$ are the generalized IPFP sequences defined in \eqref{eq:IPFPiteration}, then there exists a sequence of positive real numbers $(\lambda^n)_{n \in \NN}$ such that
\[
u^n+\lambda^n\to u  \text{ in } L^{p}(\rho_1) \quad \text{ and } \quad v^n-\lambda^n \to v  \text{ in } L^{p}(\rho_2), \, \quad 1\leq p < \infty.
\] 
where $(u,v)$ solves the dual problem \eqref{eq:dualitySep}. In particular, the sequence of couplings $\gammaep_{n}$ defined as $\gammaep_{n} = \Psi'\left((u^{n}(x)+v^{n}(y)-c(x,y))/\ep\right)$, converges in $L^{p}(\rho_1\otimes\rho_2)$ to $\gammaep_{opt}$, the density of the minimizer of \eqref{eq:mainKL} with respect to $\rho_1 \otimes \rho_2$. 
\end{theorem}
\small
\section*{Acknowledgments}

S. DM. is member of ``Gruppo Nazionale per l’Analisi Matematica, la Probabilit\`a e le loro Applicazioni" (GNAMPA). A.G. acknowledges funding by the European Research Council under H2020/MSCA-IF “OTmeetsDFT” [grant ID: 795942].


\small
\bibliographystyle{plain}
\bibliography{refs}

\cleardoublepage
\setcounter{equation}{0}
\setcounter{figure}{0}
\setcounter{table}{0}
\setcounter{page}{1}

\makeatletter

%

%
%
%
%
%
 \hsize\textwidth
    \linewidth\hsize
    \vskip -0.6in
    \@toptitlebar
    \begin{center}
    {\LARGE\bf Suplementary material: Optimal Transport losses with general convex regularization\par}
    \end{center}
    \@bottomtitlebar

\appendix
\addcontentsline{toc}{section}{Appendices}
\renewcommand{\thesubsection}{\Alph{subsection}}

\section{Section \ref{sec:bounded}: A priori estimates and $(c,\ep,\Phi)$-transforms}\label{app:apriori}

\subsection*{Proof of Lemma \ref{lemma:F1F2} $(i)$}\label{app:F1F2}
\begin{proof}
Let $u\in\Lg(\rho_1)$, $v\in\Lg(\rho_2)$ and $\Psi = \Phi^*$. By definition, we have that  $\ucep(y)$ can be defined pointwisely by
\[
\ucep (y) :=\argmmax_{t\in\RR} \, \alpha^y(t)  \qquad \text{ where } \qquad \alpha^y(t) := t  - \ep\int_{X}\Psi((u+t-c)/\ep).
\]
Notice that $\alpha^y(t)$ is concave and  $(\alpha^y)'(t_0) = 0$, if and only if, 
\[
1  = \int_{X}\Psi'((u(x)+t_0-c(x,y))/\ep)d\rho_1(x) =:  \beta^{y}(t_0). 
\]
Therefore, $\ucep(y)$ solves
\begin{equation}\label{eqn:condition}
1 = \int_{X}\Psi'((u(x)+\ucep(y)-c(x,y))/\ep)d\rho_1(x), \quad \forall~ y\in Y.
\end{equation}
Due to the boundedness of $c$, given $y',y\in Y$, $\vert c(x,y') - c(x,y) \vert \leq 2\Vert c\Vert_{\infty}$. In particular, $c(x,y')\geq c(x,y) - 2\Vert c\Vert_{\infty}$. and since $\Psi'$ is increasing, one has
\[
1  = \int_{X}\Psi'((u+\ucep(y')-c(x,y'))/\ep)d\rho_1 \leq \int_{X}\Psi'((u+\ucep(y')-c(x,y) + 2\Vert c\Vert_{\infty}))/\ep)d\rho_1.
\]
But one has also that $\beta^{y}(\ucep (y)) = 1$ and using again the monotonicity of $\beta^y(t)$, we can get that $\ucep(y') + 2\Vert c\Vert_{\infty} \geq \ucep(y)$. By exchanging $y,y'$, we then conclude $\vert \ucep(y) - \ucep(y')\vert \leq 2\Vert c\Vert_{\infty}$. Analogously, we can show the boundedness of $\vcep$.

Since $\Psi'(0)=1$, or, equivalently $\Phi'(1)=0$, that is the minimum of $\Phi$ is attained at $1$, then from \eqref{eqn:condition} one could get also
\begin{equation}\label{eqn:comparison}
\ucep(y) - \|c\|_{\infty}+ \inf u \leq 0 \qquad \ucep(y)+ \|c\|_{\infty} + \sup u \geq 0
\end{equation}

Hence, ${\rm osc}(\ucep) \leq \Vert c\Vert_{\infty}$. Similarly, we also have ${\rm osc}(\vcep) \leq 2\Vert c\Vert_{\infty}$.
\end{proof}

\subsection*{Proof of Lemma \ref{lemma:dual}}\label{app:lemmadual}

\begin{proof} 
Let $u\in\Lg(\rho_1)$ and $v\in\Lg(\rho_2)$, then
\begin{align}
\Dep(u,v) &= \int_X ud\rho_1 + \int_Y vd\rho_2 - \ep\int_{X\times Y} \Psi\left(\frac{u+v-c}{\ep}\right)d(\rho_1\otimes\rho_2),\\
&\leq  \int_X ud\rho_1 +  \max_{v\in\Lg(\rho_2)}\int_{X\times Y} v(y) - \ep \Psi\left(\frac{u+v-c}{\ep}\right) d(\rho_1\otimes\rho_2), \\
&= \int_X ud\rho_1 + \int_{X\times Y} \ucep - \ep \Psi\left(\frac{u+\ucep-c}{\ep}\right) d(\rho_1\otimes\rho_2).
\end{align}  
Therefore, for any $u\in \Lg(\rho_1)$, $ \Dep(u,v) \leq \Dep(u,\ucep)$, since the function $\alpha^y(t)= t-\ep\int_{X}\Psi((u+t-c)/\ep)d\rho_1$ is strictly concave and attains its maximum in $u^{(c,\ep,\Phi)}$. In particular, $D_{\ep}(u,\ucep)= D_{\ep}(u,v)$ if and only if $v = \ucep$.
\end{proof}

\subsection*{Proof of Lemma \ref{lemma:betterpotentials}}\label{app:betterpotentials}

\begin{proof}

Let $u \in \Lg(\rho_1)$ and $v \in \Lg(\rho_2)$. Without loss of generality we can assume $u(x_0),v(y_0)<+\infty$ and set $\tilde{u} = \vcep$. By the definition of $(c,\ep,\Phi)$-transform, we have
\[
\Dep(\tilde{u},v)\geq \Dep(u,v).
\]
 We know by lemma \eqref{lemma:F1F2} that ${\rm osc}(\tilde{u}) \leq 2 \| c \|_{\infty}$ and in particular there exists a constant $a$ such that $\|\tilde{u}-a\|_{\infty} \leq \|c\|_{\infty}$. We define $u^*=\tilde{u}-a$ and we notice that in fact $u^*= (v+a)^{(c,\ep,\Phi)}$. In order to conclude it is sufficient to notice that using \eqref{eqn:comparison} we get immediately that that $v^*=(u^*)^{(c,\ep,\Phi)}$ satisfies $\|v^*\|_{\infty} \leq 2 \| c\|_{\infty}$. In the end, using the invariance of $\Dep$ by simultaneous translation we get
 \[
 \Dep(u^*,v^*) \geq \Dep(u^*,v+a)=\Dep (\bar{u},v)\geq \Dep(u,v).
 \]
\end{proof}

\section{Section \ref{sec:existmax}: Existence of a maximizer for the dual problem and complementary slackness}\label{app:existmax}

\subsection*{Proof of Proposition \ref{prop:dual}}\label{app:dual}

\begin{proof}
We first assume that $\gamma \in \Pi(\rho_1,\rho_2)$ is a density with respect to $\rho_1\otimes\rho_2$, otherwise $\OTep(\rho_1,\rho_2) = +\infty$ and therefore the inequality is automatically verified. Then,
\begin{align}
C_{\ep}(\gamma) &= \int_{X\times Y}cd\gamma + \ep G(\gamma) - \int_{X\times Y}(u+v)d\gamma + \int_X u d\rho_1 + \int_Y v d\rho_2, \\
&= \int_X u d\rho_1 + \int_Y v d\rho_2 + \int_{X\times Y} \ep \Phi(\gamma) - (u+v-c)\gamma d(\rho_1\otimes\rho_2), \\
&\geq  \int_X u d\rho_1 + \int_Y v d\rho_2 - \ep \int_{X\times Y} \Psi\left((u+v-c)/\ep\right) d(\rho_1\otimes\rho_2), \\
&= \Dep(u,v). 
\end{align}
where we used that $\Phi$ is a convex function, $\Psi = \Phi^{*}$ and $\ep \Phi(t) + \ep \Psi(s/\ep) \geq ts$ with equality if and only if $t \in \partial \Psi(s/\ep)$ (or $s\in \ep\partial \Phi(t)$).
\end{proof}

\subsection*{Proof of Theorem \ref{prop:equiv_comp}}\label{app:mainthm}

\begin{proof} 
We will prove $ 1 \Rightarrow 2  \Rightarrow 3  \Rightarrow 4  \Rightarrow 1$.

\begin{itemize}

\item[1. $\Rightarrow$ 2.] This is a straightforward application of Lemma \ref{lemma:dual}. In fact thanks to \eqref{est:optcond} we have $D_{\ep} (u^*, \Fcep(u^*)) \geq D_{\ep}(u^*,v^*)$; however, by the maximality of $u^*,v^*$ we have also $D_{\ep}(u^*,v^*) \geq D_{\ep} (u^*, \Fcep(u^*)) $, and so we conclude that $D_{\ep}(u^*,\Fcep(u^*))=D_{\ep}(u^*,v^*)$. Thanks to \eqref{lemma:betterpotentials} we then deduce that $v^*=\Fcep(u^*)$. We can follow a similar argument to prove that conversely $u^*=\Fcep(v^*)$.

\item[2. $\Rightarrow$ 3.] A simple calculation shows for every $u \in \Lg(\rho_1) $ and $ v \in \Lg(\rho_2)$ we have
\[
(\pi_1)_{\sharp} \Psi'\left((u+v-c)/\ep\right) = \Psi'\left((u(x)+\ucep(y)-c(x,y))/\ep\right)\rho_1, \text{ and }
\] 
\[
(\pi_2)_{\sharp} \Psi'\left((u+v-c)/\ep\right) = \Psi'\left(\vcep(x)+v(y)-c(x,y))/\ep\right)\rho_2.
\]
So if we assume 2, it is trivial to see that in fact $\gamma^* =\Psi'\left((u^*+v^*-c)/\ep\right)  \in \Pi (\rho_1, \rho_2)$.

\item[3. $\Rightarrow$ 4.] since $\gamma^* \in \Pi(\rho_1, \rho_2)$, from Lemma \ref{lemma:dual} we have
\begin{align}\label{eqn:ineq1compl}C_{\ep}(\gamma^*) &\geq D_{\ep} ( u,v) \qquad  &\forall \, u \in \Lg(\rho_1), v \in \Lg(\rho_2) \\
\label{eqn:ineq2compl}
C_{\ep}(\gamma) &\geq D_{\ep} ( u^*,v^*)  &\forall \, \gamma \in \Pi(\rho_1,\rho_2). 
\end{align}
Moreover, since by definition $\gamma^*= \Psi'\left((u^*+v^*-c)/\ep\right)$, Lemma \ref{lemma:dual} assure us also that
\begin{equation}\label{eqn:ineq3compl} C_{\ep}(\gamma^*) \geq D_{\ep}( u^*,v^*).
\end{equation}
Putting now \eqref{eqn:ineq1compl}, \eqref{eqn:ineq2compl} and \eqref{eqn:ineq3compl} together we obtain
$$ C_{\ep}(\gamma^*) \geq  D_{\ep} ( u^*,v^*) =  C_{\ep}(\gamma^*) \geq  D_{\ep} ( u,v);$$
in particular we have $C_{\ep}(\gamma) \geq C_{\ep}(\gamma^*)$ which grants us that $\gamma^*$ is a minimizer for \eqref{eq:mainKL} and that in particular $\OTep(\rho_1,\rho_2) =  C_{\ep}(\gamma^*) = D_{\ep} ( u^*,v^*)$.

\item[4. $\Rightarrow$ 1.] Since for all $\gamma\in\Pi(\rho_1,\rho_2)$ and $u\in\Lg(\rho_1),v\in\Lg(\rho_2)$ we have $C_{\ep}(\gamma) \leq \Dep(u,v)$, by minimizing the left-hand side of the former inequality in $\gamma$ we find that
$$\OTep(\rho_1,\rho_2) \geq D_{\ep} ( u,v) + \ep \qquad  \forall \, u \in \Lg(\rho_1), v \in \Lg(\rho_2);$$
using that by hypotesis $\OTep(\rho_1,\rho_2)= D_{\ep} ( u^*,v^*)$, we get that
$$D_{\ep} ( u^*,v^*)  \geq D_{\ep} ( u,v)   \qquad  \forall \, u \in \Lg(\rho_1), v \in \Lg(\rho_2),$$
that is, $u^*,v^*$ are maximizing potentials for \eqref{eq:dualitySep}.

\end{itemize}
Notice that in proving $3 \Rightarrow 4$ we incidentally proved that $\gamma^*$ is the (unique) minimizer.
\end{proof}

\section{Section \ref{sec:divergences}: Properties of $\OTep$ and $\Phi$-Sinkhorn divergences}\label{app:divergence}

\subsection*{Proof of Proposition \ref{prop:OTcont} }\label{appdiv:propOTcont}

\begin{proof}
Consider the sequences $(\rho^n_1)_{n\in\NN}$ and $(\rho^n_2)_{n\in\NN}$ weakly converging respectively to $\rho_1$ and $\rho_2$. For each $n\in\NN$ consider the couple of optimal potentials $(u_n,v_n)$. We can assume without loss of generality that one is the $u_n = (v_n)^{(c,\ep,\Phi)}$ and $v_n = (u_n)^{(c,\ep,\Phi)}$ due to Theorem \ref{prop:equiv_comp} (ii). 

By proposition \ref{lemma:F1F2}, we have that, for all $n\in\NN$, $u_n$ and $v_n$ are bounded and then, by Banach-Alaoglu theorem, there exists subsequences $(u_{n_k})_{n_k\in\NN}$ and $(v_{n_k})_{n_k\in\NN}$ such that $u_{n_k}\to \overline{u}$ and $v_{n_k}\to \overline{v}$ uniformly. Finally, by arguing similar to Theorem \ref{thm:kanto2Nmax}, one can show that $(\overline{u},\overline{v})$ is a maximizer couple for $\rho_1$ and $\rho_2$. 

Now we turn to the differentiability. Let $\rho^t_1 = \rho_1 + t\chi_1$, $\rho^t_2 = \rho_2 + t\chi_2$ and consider $(u,v)$ (resp. $(u_t,v_t)$) the optimal potentials for $\OTep(\rho_1,\rho_2)$ (resp. $\OTep(\rho^t_1,\rho^t_2)$). By one hand, using $(u_t,v_t)$ as competitors for $\OTep(\rho_1,\rho_2)$, we have
\begin{align*}
\dfrac{\OTep(\rho^t_1,\rho^t_2)-\OTep(\rho_1,\rho_2)}{t} &\leq  \int_X u_t d\chi_1 - \ep \int_{X\times Y}\Psi\left(\dfrac{u_t(x)-v_t(y)-c(x,y)}{\ep}\right)d(\chi_1\otimes\rho_2) + \\
&+\int_Y v_t d\chi_2 - \ep \int_{X\times Y}\Psi\left(\dfrac{u_t(x)-v_t(y)-c(x,y)}{\ep}\right)d(\rho_1\otimes\chi_2).
\end{align*}
In particular, since $u_t\to u$ and $v_t\to v$ uniformly
\begin{align*}
\limsup_{t\to 0}&\frac{1}{t}\left(\OTep(\rho^t_1,\rho^t_2)-\OTep(\rho_1,\rho_2)\right) \leq  \int_X u d\chi_1 + \int_Y v d\chi_2  \\ 
&-\ep \int_{X}\Psi\left(\dfrac{u(x)-v(y)-c(x,y)}{\ep}\right)d(\chi_1\otimes\rho_2)-\ep \int_{Y}\Psi\left(\dfrac{u(x)-v(y)-c(x,y)}{\ep}\right)d(\rho_1\otimes\chi_2).
\end{align*}
On the other hand, if $(u,v)$ are optimal potentials for $\OTep(\rho_1,\rho_2)$ then one can also obtain a lower bound 
\begin{align*}
\dfrac{\OTep(\rho^t_1,\rho^t_2)-\OTep(\rho_1,\rho_2)}{t} &\geq  \int_X u d\chi_1  - \ep \int_{X\times Y}\Psi\left(\dfrac{u(x)-v(y)-c(x,y)}{\ep}\right)d(\chi_1\otimes\rho_2) \\
&+ \int_Y v d\chi_2  - \ep \int_{X\times Y}\Psi\left(\dfrac{u(x)-v(y)-c(x,y)}{\ep}\right)d(\rho_1\otimes\chi_2).
\end{align*}
which implies that 
\begin{align*}
\liminf_{t\to 0}&\frac{1}{t}\left(\OTep(\rho^t_1,\rho^t_2)-\OTep(\rho_1,\rho_2)\right) \geq  \int_X u d\chi_1 + \int_Y v d\chi_2  \\ 
&-\ep \int_{X}\Psi\left(\dfrac{u(x)-v(y)-c(x,y)}{\ep}\right)d(\chi_1\otimes\rho_2)-\ep \int_{Y}\Psi\left(\dfrac{u(x)-v(y)-c(x,y)}{\ep}\right)d(\rho_1\otimes\chi_2).
\end{align*}
Combining both inequalities, we conclude
\[
\frac{\delta\OTep}{\delta \rho_1}(\rho_1,\rho_2) = u - \int_{Y}\Psi\left(\frac{u+v-c}{\ep}\right)d\rho_2, \, \quad \mbox{and} \quad \, \frac{\delta\OTep}{\delta \rho_2}(\rho_1,\rho_2) = v - \int_{X}\Psi\left(\frac{u+v-c}{\ep}\right)d\rho_1.
\]

\end{proof}

\subsection*{Proof of Proposition \ref{prop:sinkdivlimits}}

\begin{proof}
The convergence proof of $\OTepd(\rho_1,\rho_2)$ to the classical OT-loss $ \OT_0(\rho_1,\rho_2)$ follows by applying the block-approximation procedure developed in \cite{CarDuvPeySch} (see Theorem 2.7 and Definition 2.9) for the Entropy-case. We omit the details here because the block-approximation holds in the same way also in our framework. Now we turn to $(b)$, we prove directly the $\Gamma$-convergence: let $\gamma^0,\gammaep\in\Pi(\rho_1,\rho_2)$ such that $\gammaep$ solves \eqref{eq:mainKL} and $\gammaep$ weakly converges to $\gamma^0$ when $\epsilon\to 0^+$. Since $\Phi$ is lower semi-continuous then $G$ is lower semi-continuous for the weak convergence and we have
\[
\liminf_{\ep\to+\infty}\OTep(\rho_1,\rho_2) = \liminf_{\ep\to+\infty}\int_{X\times Y}cd\gammaep +\ep G(\gammaep|\rho_1\otimes\rho_2)\geq \liminf_{\ep\to+\infty}\int_{X\times Y}cd\gammaep = \OT_{\infty}(\rho_1,\rho_2).
\]
Now, by taking the constant sequence $\gammaep = \rho_1\otimes\rho_2$ we have that $G(\rho_1\otimes\rho_2|\rho_1\otimes\rho_2) = 0$ and therefore  $\limsup_{\ep\to+\infty}\OTep(\rho_1,\rho_2)\leq \OT_{\infty}(\rho_1,\rho_2)$.

\end{proof}

\section{Convergence proof of the generalized Sinkhorn algorithm (Theorem \ref{thm:convIPFP})}\label{sec:convergenceIPFP}


\begin{proposition}\label{prop:EstPot} Let $(X,d_X)$, $(Y,d_Y)$ be complete separable metric spaces, $\ep>0$, $c:X\times Y\to \RR$.
If $\vert c\vert \leq M$, then $\Fcep:L^{\infty}(\rho_1)\to L^p(\rho_2)$ is a $1$-Lipschitz compact operator.
\end{proposition}
\begin{proof}
We first prove that $\Fcep$ is $1$-Lipschitz. In fact, letting $u, \tilde{u} \in L^{\infty}(\rho_1)$, we can perform a calculation very similar to what has been done in lemma \ref{lemma:F1F2} (ii):  we have that $\ucep$ and $(\tilde{u})^{(c,\ep,\Phi)}$ are such that
$$1 = \int_{X}\Psi'((u(x)+\ucep(y)-c(x,y))/\ep)d\rho_1(x), \quad \forall~ y\in Y.$$
$$1 = \int_{X}\Psi'((\tilde{u}(x)+(\tilde{u})^{(c,\ep,\Phi)}(y)-c(x,y))/\ep)d\rho_1(x), \quad \forall~ y\in Y.$$

Denote by $\beta^{y}(t) = \int_{X}\Psi'((u(x)+t-c(x,y))/\ep)d\rho_1$ and we similarly define $\tilde{\beta}^y(\tilde{t})$. Since $\Psi'$ is a increasing function, $\beta^y(t)$ is also increasing in $t$. Assume that $t_{y}$ and $\tilde{t}_{{y}}$ are such that $\beta^{y}(t_y) = \tilde{\beta}^{{y}}(\tilde{t}_{{y}})$, then we have
\begin{align*}
\tilde{\beta}^y(\tilde{t}_y) =\beta^{y}(t_y) &= \int_{X}\Psi'((u(x)+t_y-c(x,y))/\ep))d\rho_1(x) \\
&\geq \int_{X}\Psi'((\tilde{u}(x)+t_y-\Vert u-\tilde{u}\Vert_{\infty}-c(x,{y}))/\ep)d\rho_1(x) = \tilde{\beta}^{{y}}(t_{y}-\Vert u-\tilde{u}\Vert_{\infty}),
\end{align*}
which implies that $\tilde{t}_y \geq t_{{y}}-\Vert u-\tilde{u}\Vert_{\infty}$. By exchanging the roles of $t_y$ and  $\tilde{t}_{{y}}$  we conclude that $| \tilde{t}_{\tilde{y}} -  t_y| \leq \Vert u-\tilde{u}\Vert_{\infty}$. Now we can consider $t_y=\Fcep(u)(y)$ and $\tilde{t}_{y}=\mathcal{F}^{(c,\ep,\Phi)}(\tilde{u})(y) $, and integrating over $y$ we get precisely that $\Fcep$ is a $1$-Lipschitz operator from $L^{\infty}(\rho_1)$ to $L^p(\rho_1)$. This proves in particular that $\Fcep: L^{\infty}(\rho_1) \to L^p(\rho_2)$ is continuous. In order to prove that $\Fcep$ is compact it suffices to prove that $\Fcep(B)$ is precompact for every bounded set $B \subset L^{\infty}(\rho_1)$. We will use Proposition 5.1 in \cite{DMaGer19}; thanks to \eqref{eqn:comparison} for sure we have that is $B$ is bounded then $\Fcep(B)$ is bounded in $L^{\infty}(\rho_2)$.

We fix $u \in L^{\infty}(\rho_1)$, $\| u\|_{\infty} \leq H$ and we consider $v=\Fcep (u)$.

First of all we consider the following two properties, which are always true thanks to \eqref{eqn:comparison} and the fact that $\Psi'$ is increasing and nonnegative
\begin{equation}\label{eqn:property1}
     |u(x)+ v(y) - c(x,y)| \leq H+ \|c \|_{\infty} \quad , \quad 0 \leq \Psi' \Bigl( \frac{u(x)+ v(y) - c(x,y)}{\ep} \Bigr) \leq M;
\end{equation}
 we can choose for example $M= \Psi' ( (H+\|c\|_{\infty}) /\ep ) $.

Let us denote $\gamma = \rho_1\otimes\rho_2$. Since $c \in L^{\infty}(\gamma)$, by Lusin theorem we have that for every $\sigma>0$ there exists $N_{\sigma} \subset X \times Y$,  with $\gamma ( N_{\sigma})< \sigma$, such that $c|_{(N_{\sigma})^c}$ is uniformly continuous, with modulus of continuity $\omega_{\sigma}$. In particular there exists $\phi \in C(X \times Y)$ which is $\omega_{\sigma}$-continuous that extends $c|_{(N_{\sigma})^c}$; we now consider the slices of $N_{\sigma}$
$$N_y=\{ x \in X \; : (x,y) \in N_{\sigma} \} \qquad G_y = X\setminus N_y.$$
In what follows $N_y$ will be considered the bad points while $G_y$ are the good points (where $c$ coincides with $\phi$). Since by Fubini
$$\gamma(N_{\sigma}) = \int_Y \rho_1(N_y) \, d \rho_2(y),$$
we deduce that the set $Y_g \subset Y$ where $\rho_1(N_y) < \sigma^{1/2}$ has measure at least $1-\sigma^{1/2}$. From now on we consider $y,y' \in Y_g$ and we can now estimate the oscillation of $u (y)$: when we subtract the two optimality conditions in $y$ and $y'$ we get
\begin{equation}\label{eqn:difference}
    \int_X \Psi'( (v(y) + u(x) - c(x,y))/\ep ) - \Psi'(v(y') + u(x) - c(x,y'))/\ep)\, d \rho_1(x) =0
\end{equation} 

Now we multiply and divide inside by $v(y)-c(x,y) -v(y') + c(x,y')$ and we denote by 

$$K^{y,y'}(x) = \frac{\Psi'\bigl( \frac{v(y) + u(x) - c(x,y)}{\ep} \bigr) - \Psi'\bigl( \frac{v(y') + u(x) - c(x,y')}{\ep} \bigr)}{v(y)-c(x,y) -v(y') + c(x,y')} \geq 0, $$

where the positivity is granted thanks to the fact that  $\Psi'$ is increasing.

Now we split the estimate in two cases, according to some threshold $\alpha$ to be decided later. We denote 
$A_{\alpha} =\{x \in X \; : \; |u(y)-u(y') - c(x,y) +c(x,y')| \leq \alpha\}$

\begin{itemize}
    \item If $\rho_1(A_{\alpha}) > 2\sigma^{1/2}$ then we can say that
    $A_{\alpha} \cap G_y \cap G_y' \neq \emptyset$ and so there exists $x \in X$ such that $(x,y), (x,y') \in N_{\sigma}$ and $x \in A_{\alpha}$, in particular
$$|u(y) - u(y')| \leq | c(x,y)-c(x,y')| + \alpha \leq \omega_{\sigma}(y,y') + \alpha.$$
\item $\rho_1(A_{\alpha}) \leq 2 \sigma^{1/2}$. In this case we want to estimate the integral in \eqref{eqn:difference}; we split the integral inside $A_{\alpha} \cup N_y \cup N_{y'}$ and in  $B_{\alpha}:=X \setminus ( A_{\alpha} \cup N_y \cup N_{y'}) $. We write, using the definition of $K^{y,y'}(x)$,
$$ \int_{B_{\alpha}} (v(y)-c(x,y) -v(y') + c(x,y')) K^{y,y'}(x) \, d \rho_1(x) =\int_{A_{\alpha} \cup N_y \cup N_{y'}} ( \ldots ) \, d \rho_1, $$
$$(v(y) - v(y'))\int_{B_{\alpha}} K^{y,y'}(x) \, d \rho_1(x) = \int_{B_{\alpha}} (c(x,y)- c(x,y'))K^{y,y'}(x) \, d \rho_1+ \int_{A_{\alpha} \cup N_y \cup N_{y'}} ( \ldots ) \, d \rho_1, $$
$$(v(y) - v(y')) = \frac{\int_{B_{\alpha}} (c(x,y)- c(x,y'))K^{y,y'}(x)}{\int_{B_{\alpha}} K^{y,y'}(x) \, d \rho_1(x)} \, d \rho_1+ \frac{\int_{A_{\alpha} \cup N_y \cup N_{y'}} ( \ldots ) \, d \rho_1}{\int_{B_{\alpha}} K^{y,y'}(x) \, d \rho_1(x)}. $$
where with $(\ldots)$ we mean the integrand in \eqref{eqn:difference}, which we know already to be smaller than $M$. Now we know that in $B_{\alpha}$ we have $c(x,y)=\phi(x,y)$ and $c(x,y') = \phi(x,y')$, in particular, using the $\omega_{\sigma}$-continuity of $\phi$, we obtain
\begin{equation}\label{eqn:vv}
|v(y) - v(y')| \leq \omega_{\sigma} (y,y') + \frac{\rho_1( A_{\alpha} \cup N_y \cup N_{y'} ) M}{\int_{B_{\alpha}} K^{y,y'}(x) \, d \rho_1(x)} \leq  \omega_{\sigma} (y,y') + \frac{4 \sigma^{1/2} M}{\int_{B_{\alpha}} K^{y,y'}(x) \, d \rho_1(x)}
\end{equation}

In order to conclude it is sufficient to get an estimate from above of $\int_{B_{\alpha}} K^{y,y'}(x) \, d \rho_1(x)$. Using the fact that $\Psi'$ is strictly increasing when valued in $(1-\eta, +\infty)$, we have (by compactness) that $$\delta(\alpha):=\inf\{ |\Psi'(s)-\Psi'(t)| \; : \; |s-t|\geq \alpha, \; 1-\tfrac{\eta}2  \leq \Psi'(s) \leq  M \} >0. $$
Moreover, since $\int_X \Psi' ( \frac{u+v-c}{\ep} ) \, d \rho_1 =1$ and $\Psi'(\frac{u+v-c}{\ep})<M$, we get, for every $y \in Y$ that 
$$C_y:= \left\{x \; : \; \Psi' \Bigl( \frac {u(x)+v(y)-c(x,y)}{\ep}\Bigr) \geq 1-\frac{\eta}2\right\} \quad \Longrightarrow \quad \rho_1(C_y) \geq \frac{\eta}{2M}. $$
Using the definitions of $C_y$ and $\delta(\alpha)$ we get that $K^{y,y'}(x) \geq \frac{\delta(\alpha)}{4\|c\|_{\infty}}$ for $x \in C_y \cap B_{\alpha}$. Since $\rho_1(B_{\alpha})\geq 1- 4 \sigma^{1/2}$ we have $\rho_1( C_y \cap B_{\alpha})\geq \frac{\eta}{2M} - 4 \sigma^{1/2} \geq \frac {\eta}{4M}$ for sufficiently small $\sigma$. In particular we have

$$\int_{B_{\alpha}} K^{y,y'}(x) \, d \rho_1 \geq \frac {\delta (\alpha)}{4\|c \|_{\infty}} \cdot \rho_1( B_{\alpha} \cap C_y) \geq \frac { \delta(\alpha) \eta}{16 \|c\|_{\infty} M }. $$

Plugging this estimate into \eqref{eqn:vv} we obtain
\begin{equation}\label{eqn:vvII}
    |v(y)-v(y')| \leq \omega_{\sigma} ( y,y') + C\cdot\frac{ \sigma^{1/2}}{\delta(\alpha)},
\end{equation}

where the constant $C$ depends only on $\Psi'$ and $\|c\|_{\infty}$ but not on $u$ or $v$.

\end{itemize}
Summing it up we then obtain that if $y,y' \in Y_g$ we have that 
$$ |v(y)-v(y')| \leq \omega_{\sigma} (y,y') + \max \left\{ C \frac {\sigma^{1/2}}{\delta(\alpha)} , \alpha \right\}.$$
It is clear that choosing $\alpha$ and then $\sigma$ we can apply Proposition 5.1 in \cite{DMaGer19} to conclude.

\end{proof}

\begin{proof}[Proof of Theorem \ref{thm:convIPFP}]
Let $(u^n)_{n\in\NN}$ and $(v^n)_{n\in\NN}$ be the IPFP sequence defined in \eqref{eq:IPFPiteration}. One can rewrite it with the help of the $(c, \ep,\Phi)$-transform:
\[
\begin{cases}
v_{2n+1} = (u_{2n})^{(c,\ep,\Phi)} \\
u_{2n+1} = u_{2n} \\ 
     \end{cases}, \quad \begin{cases}
v_{2n+2} =  v_{2n+1} \\
u_{2n+2} = (v_{2n+1})^{(c,\ep,\Phi)} \\ 
     \end{cases}.
\] 	

Notice that, as soon as $n\geq2$, we have $u_n \in L^{\infty}(\rho_1)$ and $v_n \in L^{\infty}(\rho_2)$ thanks to the regularizing properties of the $(c,\ep,\Phi)$-transforms proven in Lemma \ref{lemma:F1F2} and, moreover, thanks to \eqref{est:optcond}  and Proposition \ref{prop:dual} we have 
\[
\Dep(u_n,v_n)\leq \Dep(u_{n+1},v_{n+1}) \leq \dots \leq \OTep(\rho_1,\rho_2).
\]
Then, by the same argument used in the proof of Lemma \ref{lemma:betterpotentials} it is easy to prove that there for each $n \geq 2$ there exists $\ell_n \in \R$ such that $ \|u_n - \ell_n \|_{\infty}, \|v_n +\ell_n\| \leq 2 \|c\|_{\infty}$. Now, thanks to Proposition \ref{prop:EstPot} we have that the sequeces $u_n - \ell_n$ and $v_n +\ell_n$ are precompact in every $L^p$, for $1 \leq p < \infty$; in particular let us consider any limit point $u,v$. Then we have a subsequence $u_{n_k},v_{n_k}$ such that $u_{n_k}\to u$,$v_{n_k}\to v$ in $L^{\infty}$ and $u_{n_k+1} = (v_{n_k})^{(c,\ep)}$ (or the opposite). Using the continuity in $L^p$ of the $(c,\ep,\Phi)$-transforms, and the fact that an increasing and bounded sequence has vanishing increments, we obtain
\[
\Dep(\vcep,v) - \Dep(u,v) = \lim_{n_k\to\infty} \Dep(u_{n_k+1},v_{n_k+1}) - \Dep({u_{n_k}},v_{n_k}) = 0.
\]   
In particular, by \eqref{est:optcond}, we have $u=\vcep$. Analogously, we obtain that $v=\ucep$ by doing the same calculation using the potentials $(u_{n_{k+2}}, v_{n_{k+2}})$ and then 
\[
\Dep(u,\ucep) - \Dep(u,v) = \lim_{n_k\to\infty} \Dep(u_{n_k+2},v_{n_k+2}) - \Dep({u_{n_k}},v_{n_k}) = 0.
\]
Now we can use Theorem \ref{prop:equiv_comp}: the implication $2 \Rightarrow 1$ proves that $(u,v)$ is a maximizer\footnote{in order to prove that there is a unique limit point at this stage, it is sufficient to take $\ell_n$ that minimizes $\| u_n - \ell_n - u\|_2$.} and we get the convergence result for $u^n+\lambda^n$ and $v^n-\lambda^n$, setting $\lambda^n = \ell_n$.

In order to prove also the convergence of the plans, it is sufficient to note that for free we have $u_n-\ell_n+v_n+\ell_n = u_n+v_n \to u+v$ in $L^p(\rho_1 \otimes \rho_2)$, since now the translations are cancelled. Again, the fact that the $\Psi$ is Lipschitz on bounded domains and the boundedness of $k$, will let us conclude that in fact $\gamma^n \to \gamma$ in $L^p(\rho_1 \otimes \rho_2)$ for every $1 \leq p < \infty$.
\end{proof}

\section{Wasserstein Barycenter and Multi-marginal formulation}\label{app:multi}

For the sake of completeness, we briefly introduce the multi-marginal formulation of the problem \eqref{eq:mainKL} introduced in the main text. 

Let $X_1,\dots,X_N$ be complete separable metric spaces, $c_N:X_1\times\dots\times X_N\to\R$ be a bounded cost function and, for all $i\in \lbrace 1,\dots,N \rbrace$, $\Phi$ be an Entropy and let $\rho_i\in\Pro(X_i)$ be probability measures. The multi-marginal optimal transport problem with convex-regularization is defined by
\begin{equation}\label{app:KLN}
\OTNep(\rho_1,\dots,\rho_N) = \min_{\gamma^N\in\Pi_N(\rho_1,\dots,\rho_N)}\int_{X_1\times\dots\times X_N}c_Nd\gamma^N + \ep\int_{X_1\times \dots\times X_N}\Phi\left(\dfrac{d\gamma^N}{d\rho^N}\right) d\rho^N,   
\end{equation}
where we denoted $\rho^N = \otimes^N_i\rho_i$ and $\Pi_N(\rho_1,\dots,\rho_N)$ is the set of probabilities $\gamma^N \in \Pro(X_1\times\dots\times X_N)$ having $i$-th marginal equal to $\rho_i$. The dual formulation of \eqref{app:KLN} reads
\begin{align}\label{eq:dualN}
\OTNep(\rho_1,\dots,\rho_N) &= \max\lbrace \DNep(u_1,\dots,u_N) \, :  \, u_i \in \Lg(\rho_i), \, i\in\lbrace 1,\dots,N\rbrace \rbrace \\
&:= \max_{u_1,\dots,u_N\in \Lg}\left\lbrace \sum^N_{i=1}\int_{X_i}u_id\rho_i - \ep\int_{X_1\times\dots\times X_N}\Psi\left(\frac{\sum^N_{i=1}u_i-c_N}{\ep}\right)d\rho^N\right\rbrace.
\end{align}
Notice that if $c_N(x_1,\dots,x_N) = \sum_{1\leq i<j\leq N}\vert x_i-x_j\vert^2$ and $\ep=0$ then \eqref{app:KLN} is equivalent to the Wasserstein Barycenter \cite{AguCar,BenCarCutNenPey2015, CutDou2014, GaSw}.

The main idea in order to extend our results in this setting is to consider the $N=2$ marginal problem with $X=(X_1,\rho_1)$, $Y=(X_2\times\dots\times X_N,\otimes^N_{i=2}\rho_i)$ and derive analogous properties as in Lemma \ref{lemma:F1F2} and \ref{lemma:betterpotentials} of the $(N,c,\ep,\Phi)$-transforms  $\Ficep:\Lg(\otimes^N_{j\neq i}\rho_j)\to \Lg(\rho_i)$ defined below 
\begin{equation}\label{eq:cepN}
\Ficep(\hat{u}_i) \in \argmmax \lbrace \Dep(u_1,\dots,u_N) : u_i \in \Lg(\rho_i) \rbrace.
\end{equation}
where we denoted by $\hat{u}_i = (u_1,\dots,u_{i-1},u_{i+1},\dots,u_N), \, \forall i\in\lbrace 1,\dots,N\rbrace$. 

In particular, when $\Phi$ is the Shannon entropy, \eqref{eq:cepN} simply reads
\[
\Ficep (\hat{u}_i)= -\ep\ln\left(\int_{\Pi^N_{j\neq i}X_j}\exp\left(\frac{\sum_{j\neq i}u_j-c_N}{\ep}\right)d\left(\underset{j\neq i}{\overset{N}{\otimes}}\rho_j\right)\right), \, \forall i \in \lbrace 1,\dots, N\rbrace.   
\]

The prove of the next theorem follow the same lines of Theorem \ref{prop:equiv_comp}. 

\begin{theorem}[Equivalence and complementarity condition]\label{teo:equiv_compN}
Let $\ep>0$ be a positive number, $X_1,\dots,X_N$ be complete separable metric spaces, $c_N:X_1\times\dots\times X_N\to \mathbb{R}$ be a bounded cost function, $\Phi$ be an Entropy, $\Psi = \Phi^*$, $\rho_i \in \mathcal{P}(X_i), \forall i \in\lbrace 1,\dots,N\rbrace$, be probability measures. Then given $u_i^* \in \Lg(\rho_i), i\in\lbrace 1,\dots,N\rbrace$, the following are equivalent:
\begin{enumerate}
\item \emph{(Maximizers)} $u_1^*,\dots,u_N^*$ are maximizing potentials for \eqref{eq:dualN};
\item \emph{(Maximality condition)} $\Fcep_i (u_i^*)=u_i^*$ and $u_i^*\in L^{\infty}(\rho_i), \, \forall \, i\in\lbrace 1,\dots,N\rbrace$; 
\item \emph{(Minimizer)} let $\gamma^*= \Psi'\left((\sum^N_{i=1}u_i^*(x_i)-c_N(x_1,\dots,x_N))/\ep\right) \cdot \rho^N$, then $\gamma^* \in \Pi_N(\rho_1, \dots, \rho_N)$;
\item \emph{(Complementary Slackness)} $\OTNep(\rho_1,\dots,\rho_N) = D^N_{\ep} (u_1^*,\dots,u_N^*)$.
\end{enumerate}
Moreover in those cases $\gamma^*$, as defined in 3, is also the (unique) minimizer for the problem \eqref{eq:mainKL}.
\end{theorem}

\subsection*{Generalized multi-marginal Sinkhorn algorithm} Analogously to \eqref{eq:IPFPiteration}, define recursively the sequences $(u^n_j)_{n\in\NN}, j\in \lbrace 1,\dots,N\rbrace$ (Sinkhorn iterates) by
\begin{equation}\label{eq:IPFPsequenceN}
\begin{array}{lcl}
u_1^0(x_1) & = & 1, \\
u^0_j(x_j) & = & 1, \quad j \in \lbrace 2,\dots,N \rbrace, \\
u^n_j(x_j) & = &\argmmax_{u\in\Lg(\rho_j)}\left\lbrace \int_{X_j}u d\rho_j -\ep\int_{\Pi^{N}_{i\neq j}X_i}\Psi\left(\frac{\sum^{j-1}_{i=1}u^{n}_i+u+\sum^N_{i=j+1}u^{n-1}_i-c_N}{\ep}\right)d\left(\otimes^N_{i\neq j}\rho_i \right) \right\rbrace.
\end{array}
\end{equation}
We observe that $u^n_i(x_i) = \Ficep(\hat{u}^n_i)(x_i), \, \forall \, i$. In the particular case when $\Phi$ is the Shannon entropy, one can rewrite the Sinkhorn sequences \eqref{eq:IPFPsequenceN} more explicitly by
\begin{align*}
u^n_j(x_j) &= - \ep\log\left(\int_{\Pi_{i \neq j}X_i}\underset{i<j}{\otimes}e^{u^n_i(x_i)/\ep}\underset{i>j}{\otimes}e^{u^{n-1}_i(x_i)/\ep}e^{-c_N(x_1,\dots,x_N)/\ep} d\left(\otimes^N_{i\neq j}\rho_i\right)\right).
\end{align*}

Finally, we state the convergence of the generalized Sinkhorn algorithm in the multi-marginal case. The proof of the theorem is omitted since it follows similarly to the method applied in Theorem \ref{thm:convIPFP}.
\begin{theorem}\label{thm:convIPFPNmarg}
Let $(X_1,d_1), \dots, (X_N,d_N)$ be complete separable metric spaces, $\rho_1,\dots,\rho_N$ be probability measures in $X_1,\dots,X_N$, $c_N:X_1\times\dots\times X_N\to[0,+\infty]$ be a bounded cost, $\Phi$ be a Entropy function, $p$ be an integer $1\leq p <\infty$. If $(u^n_j)_{n\in\NN}, j\in\lbrace 1,\dots,N\rbrace$ are the generalized Sinkhorn sequences defined in \eqref{eq:IPFPsequenceN}, then there exist a sequence $\lambda^n \in \R^N$, with $\lambda^n_i >0$ and $\sum_{i=1}^N \lambda^n_i = 0$ such that
\[
\forall \, j\in\lbrace 1,\dots,N\rbrace, \quad  u^n_j + \lambda^n_j\to u_j \text{ in } L^p(\rho_j),
\]
where $(u_j)_{j=1}^N$ solve dual problem \eqref{eq:dualN}. In particular, the sequence $(\gamma^n)_{n\in\NN}$,
\[
\gamma^n = \Psi'\left(\dfrac{\sum^N_{i=1}u^n_i(x_i)-c_N(x_1,\dots,x_N)}{\ep}\right), 
\]
converges in  $L^p(\rho_1\otimes\dots\otimes\rho_N)$ to the optimizer $\gammaep_{opt}$  in \eqref{app:KLN}. 
\end{theorem}

\end{document}